\documentclass[a4paper, 11pt]{article}
\usepackage[fontsize=11pt]{scrextend}

\usepackage{amsmath,amssymb,amscd,xspace,fancyhdr,color,authblk,srcltx,fontenc,amsxtra,latexsym}
\usepackage{sectsty}
\usepackage{xcolor}
\sectionfont{\color{teal}}

\renewcommand\thesection{\arabic{section}}

\usepackage{stix}
\usepackage[mathscr]{eucal}

\usepackage{ntheorem}
\theoremseparator{.}
\usepackage{indentfirst}
\usepackage{hyperref}

\usepackage{scalerel}
\usepackage[usestackEOL]{stackengine}
\def\dashint{\,\ThisStyle{\ensurestackMath{\stackinset{c}{.2\LMpt}{c}{.5\LMpt}{\SavedStyle-}{\SavedStyle\phantom{\int}}}\setbox0=\hbox{$\SavedStyle\int\,$}\kern-\wd0}\int}

\newcounter{teller}

\newcounter{tellerr}

\newenvironment{tabeleq}{\begin{list}%
{\rm  (\roman{tellerr})\hfill}{\usecounter{tellerr} \leftmargin=1.1cm
	\labelwidth=1.1cm \labelsep=0cm \parsep=0cm}
}{\end{list}}

\newcounter{tellerrr}

\makeatletter
\def\eqnarray{\stepcounter{equation}\let\@currentlabel=\theequation
\global\@eqnswtrue
\tabskip\@centering\let\\=\@eqncr
$$\halign to \displaywidth\bgroup\hfil\global\@eqcnt\z@
$\displaystyle\tabskip\z@{##}$&\global\@eqcnt\@ne
\hfil$\displaystyle{{}##{}}$\hfil
&\global\@eqcnt\tw@ $\displaystyle{##}$\hfil
\tabskip\@centering&\llap{##}\tabskip\z@\cr}

\def\endeqnarray{\@@eqncr\egroup
\global\advance\c@equation\m@ne$$\global\@ignoretrue}

\def\@yeqncr{\@ifnextchar [{\@xeqncr}{\@xeqncr[5pt]}}
\makeatother

\newtheorem{theorem}{Theorem}[section]

\newtheorem{definition}[theorem]{Definition}
\newtheorem{lemma}[theorem]{Lemma}
\newtheorem{proposition}[theorem]{Proposition}

\newenvironment{proof}[1][Proof]{\textbf{#1.} }{\hfill\rule{0.5em}{0.5em}}
{\catcode`\@=11\global\let\AddToReset=\@addtoreset
\AddToReset{equation}{section}
\renewcommand{\theequation}{\thesection.\arabic{equation}}
\AddToReset{theorem}{section}

\newcommand{\tb}{\displaystyle\dashint}

\newcommand{\tp}{\displaystyle\int}

\newcommand{\su}{\mathop{\sup} \limits}

\newcommand{\ngang}{\overline}
\newcommand{\un}{\overline{u}}
\newcommand{\xn}{\ngang{x}}
\newcommand{\nga}{\widetilde}

\newcommand{\hty}{\rightharpoonup}
\newcommand\htu[1]{\xrightarrow{#1}}
\newcommand\lsu[1]{\mathop {\lim \sup }\limits_{#1}}
\newcommand{\hoi}{\longrightarrow}
\newcommand\li[1]{\mathop {\lim}\limits_{#1}}

\newcommand{\N}{\mathbb N}

\newcommand{\R}{\mathbb R}

\newcommand{\Cc}{C^\vc_{\operatorname{c}}}
\newcommand{\ax}{\mapsto} 
\newcommand{\va}{\varphi}
\newcommand{\vae}{\varepsilon}

\newcommand{\vc}{\infty}

\newcommand{\q}{\omega}

\newcommand{\1}{\mathbf{1}}

\newcommand{\di}{\operatorname{div}}

\newcommand\norm[1]{\left\lVert#1\right\rVert}

\title{{\bf Weighted gradient estimates for the class of 
\\singular $p$-Laplace system}}
\author{Tan Duc Do$^{1}$, Le Xuan Truong$^2$, Nguyen Ngoc Trong$^{3,\star}$ \\ {\small $^1$University of Economics Ho Chi Minh City \\   Email: \texttt{tanducdo.math@gmail.com}\medskip\\
{\small $^2$University of Economics Ho Chi Minh City \\ Email: \texttt{lxuantruong@gmail.com}} \medskip\\ {\small $^3$ University of Economics Ho Chi Minh City \\  $^{\star}$ Corresponding author\\ Email: \texttt{trongnn37@gmail.com}}}}
\date{\today}

\begin{document}

\maketitle
\begin{abstract}

Let $n \in \{2, 3, 4, \ldots\}$, $N \in \{1, 2, 3, \ldots\}$ and $p \in \big(1, 2-\frac{1}{n}\big]$.
Let $\beta \in (1,\infty)$ be such that
\[
\frac{np}{n-p}<\beta'<\frac{n}{n(2-p)-1}
\]
and $f \in L^{\beta}(\R^n;\R^N)$. 
Consider the $p$-Laplace system
\[
-\Delta_p u=-\di\big(|Du|^{p-2}Du\big)=f \quad \mathrm{in} \quad \R^n.
\]
We obtain a weighted gradient estimate for distributional solutions of this system.

\medskip

\noindent 

\medskip



\end{abstract}   

\tableofcontents
\section{Introduction}

Calderon-Zygmund theory is undoubtedly classical to linear partial differential equations. 
In the last few years, its extension to non-linear settings has become an active area of research. 
For a comprehensive survey on this account, cf.\ \cite{Min2} and also the references therein.
Our paper continues this trend with a gradient estimate for the solutions of a $p$-Laplace system.

Specifically, let $n \in \{2, 3, 4, \ldots\}$, $N \in \{1, 2, 3, \ldots\}$ and $p \in \big(1, 2-\frac{1}{n}\big]$.
Consider the $p$-Laplace system
\begin{equation}\label{5hh070120148}
-\Delta_p u=-\di\big(|Du|^{p-2}Du\big)=f \quad \mathrm{in} \quad \R^n,
\end{equation}
where $f: \R^n \longrightarrow \R^N$ belongs to some appropriate Lebesgue space.

Our aim is to derive a general Muckenhoupt-Wheeden-type gradient estimate for \eqref{5hh070120148}. 
This result inherits the spirit of \cite{KM}, \cite{NP}, \cite{NP2} and \cite{NP3}.
Specifically, in \cite{NP}, \cite{NP2} the authors obtained such estimates when $N=1$ and $1<p\leq 2-\frac{1}{n}$.
If in addition $\frac{3n-2}{2n-1}<p\leq 2-\frac{1}{n}$, pointwise gradient estimates with measure data are also available (cf.\ \cite{NP3}).
In a system setting (i.e.\ $N \ge 1$) with measure data, pointwise grandient bounds via Riesz potential and Wolff potential for $p>2-\frac{1}{n}$ were obtained in \cite{KM}.
Regarding the method of proof, we follow the general frameworks presented in these papers. 
Our main contribution involves the reconstructions of a comparison estimate and a good-$\lambda$-type bound peculiar to the setting in this paper.

To state our main result, we need some definitions.

\begin{definition}
A function $u:\R^n \to \R^N$ is a distributional (or weak) solution to \eqref{5hh070120148} if
\[
\int_{\R^n} |Du|^{p-2}Du:D\va \, dx=\int_{\R^n} f \va \, dx 
\]
for all $ \va \in \Cc(\R^n,\R^N)$.
\end{definition}

Here $Du$, which is a counterpart of $\nabla u$ in the equation setting, is understood in the sense of tensors.
See Section \ref{tensor} for further details.

Next recall the notion of Muckenhoupt weights.

\begin{definition}
A positive function $\q \in L^1_{\text{loc}}(\mathbb{R}^{n})$ is said to be an $\mathbf{A}_{\infty}$-weight if there exist constants $C > 0$ and $\nu > 0$ such that
$$
\q(E)\le C \left(\frac{|E|}{|B|}\right)^\nu \q(B),
$$
for all balls $B \subset \R^n$ and all measurable subset $E$ of $B$. The pair $(C,\nu) $ is called the $\mathbf{A}_\infty$-constants of $\q$ and is denoted by $[\q]_{\mathbf{A}_\infty}$.  
\end{definition}

In what follows, we will also make use of the maximal function defined by
\[
{\bf M}_\beta(f)(x)
=\sup_{\rho>0} \rho^\beta \, \tb_{B_\rho(x)}|f(y)|dy
\]
for all $x\in\mathbb{R}^{n}$, $f \in L^1_{\mathrm{loc}}(\R^n)$ and $\beta \in [0,n]$, where  
$$
\tb_{B_\rho(x)}|f(y)| \, dy := \frac{1}{|B_\rho(x)|} \int_{B_\rho(x)} f(y) \, dy.
$$
When $\beta = 0$, the Hardy-Littlewood maximal function ${\bf M} = {\bf M}_0$ is recovered.

Our main result is as follows.

\begin{theorem} \label{101120143-p} 
Let $n \in \{2, 3, 4, \ldots\}$, $N \in \{1, 2, 3, \ldots\}$ and $p \in \big(1, 2-\frac{1}{n}\big]$.
Let $\beta \in (1,\infty)$ be such that
\[
\frac{np}{n-p}<\beta'<\frac{n}{n(2-p)-1}
\]
and $f \in L^{\beta}(\R^n;\R^N)$. 
Let $\Phi: [0,\infty)\rightarrow [0,\infty)$ be a strictly increasing function that satisfies
\[
\Phi(0)=0
\quad \mbox{and} \quad
\lim_{t\rightarrow\infty}\Phi(t)=\infty.
\]
Furthermore assume that there exists a $c > 1$ such that
\[
\Phi(2t)\leq c\, \Phi(t)
\]
for all $t\geq 0$.
Then for all $\q\in \mathbf{A}_{\infty}$ there exist a $C > 0$ and a $\delta \in (0,1)$, both depending on $n$, $p$, $\Phi$ and $[\q]_{\mathbf{A}_{\infty}}$ only, such that                             
\[
\int_{\R^n}\Phi(|D u|) \, \q \, dx \leq C \int_{\R^n} \Phi\left[\Big(\mathbf{M}_\beta\big(|f|^{\beta}\big)\Big)^{\frac{1}{(p-1)\beta}}\right] \q \, dx
\]
for all distributional solution $u$ of \eqref{5hh070120148}.
\end{theorem}

Note that in our setting all functions are vector fields. 
For short we will write, for instance, $\Cc(\R^n)$ in place of $\Cc(\R^n,\R^N)$ hereafter. 
When scalar-valued functions are in use, we will explicitly write $\Cc(\R^n,\R)$.
This convention applies to all function spaces in the whole paper.

When $n=1$ it has been known that the distributional solution $u$ is locally $C^{1,\alpha}$ for some exponent $\alpha = \alpha(n,N,p) > 0$, whose result is due to \cite{Uhl}.
Hence we only consider $n \ge 2$ in this project.
We also remark that the function $\Phi$ in the above theorem is quite general. In particular, we do not require $\Phi$ to be convex or to satisfy the so-called $\nabla_2$ condition: $\Phi(t)\geq \frac{1}{2a}\Phi(at)$ for some $a>1$ and for all $t\geq 0$. As such one can take, for examples, $\Phi(t)=t^\alpha$ or $\Phi(t)=[\log(1+t)]^{\alpha}$ for any $\alpha>0$.

The outline of the paper is as follows.
Section \ref{tensor} collects definitions and basic facts about tensors and $p$-harmonic maps. 
In Sections \ref{compare} and \ref{sec-3} we derive a comparison estimate and a good-$\lambda$-type bound respectively.
Lastly Theorem \ref{101120143-p} is proved in Section \ref{main proof}.

{\bf Notations.} \quad
Throughout the paper the following set of notation is used without mentioning.
Set $\N = \{0, 1, 2, 3, \ldots\}$ and $\N^* = \{1, 2, 3, \ldots\}$.
For all $a, b \in \R$, $a \wedge b = \min\{a,b\}$ and $a \vee b = \max\{a,b\}$.
For all ball $B \subset \R^d$ we write $w(B) := \int_B w$.
The constants $C$ and $c$ are always assumed to be positive and independent of the main parameters whose values change from line to line.
Given a ball $B = B_r(x)$, we let $t B = B_{tr}(x)$ for all $t > 0$.
If $p \in [1,\infty)$, then the conjugate index of $p$ is denoted by $p'$.

{\bf Throughout assumptions.} \quad 
In the entire paper, we always assume that $n \in \{2, 3, 4, \ldots\}$, $N \in \{1, 2, 3, \ldots\}$ and $p \in \big(1, 2-\frac{1}{n}\big]$ without explicitly stated.

\section{Tensors and $p$-harmonic maps} \label{tensor}

This section briefly summarizes definitions and basic facts regarding tensors and $p$-harmonic maps.
Further details are available in \cite[Sections 2 and 3]{KM}.
These will be used frequently in subsequent sections without mentioning.

Let $\{e_j\}_{j=1}^n$ and $\{e^{\alpha}\}_{\alpha=1}^N$ be the canonical bases of $\R^n$ and $\R^N$ respectively.
Let $\zeta$ and $\xi$ be second-order tensors of size $(N,n)$, that is,
\[
\zeta = \zeta_j^\alpha \, e^\alpha \otimes e_j
\quad \mbox{and} \quad
\xi = \xi_j^\alpha \, e^\alpha \otimes e_j
\]
in which repeated indices are summed.
Note that the linear space of all second-order tensors is isomorphic to $\R^{N \times n}$.

The Frobenius product of $\zeta$ and $\xi$ is given by
\[
\zeta : \xi = \zeta_j^\alpha \, \xi_j^\alpha,
\]
from which we also obtain the Frobenius norm of $\zeta$ as
$
|\zeta|^2 = \zeta : \zeta.
$
The divergence of $\zeta$ is defined by
\[
\mathrm{div} \, \zeta = (\partial_j \zeta_j^\alpha) \, e^\alpha.
\]
Also the gradient of a first-order tensor $u = u^\alpha \, e^\alpha$ is the second-order tensor 
\[
Du = (\partial_j u^\alpha) \, e^\alpha \otimes e_j.
\]

Next consider the tensor field
\[
A_q(z) := |z|^{q-2} \, z = |z|^{q-2} \, z_j^\alpha \, e^\alpha \otimes e_j
\]
defined on the linear space of all second-order tensors, where $q \in (1,\infty)$.
The differential of $A_q$ is defined as a fourth-order tensor 
\[
\partial A_q(z) 
= |z|^{q-2} \, \left( \delta_{\alpha\beta} \, \delta_{ij} 
+ (q-2) \, \frac{z_i^\alpha \, z_j^\beta}{|z|^2} \right) \, 
(e^\alpha \otimes e_i) \otimes (e^\beta \otimes e_j).
\]
Here $\delta_{\alpha\beta}$ is the Kronecker's delta.
This leads to
\[
\partial A_q(z) : \xi 
= |z|^{p-2} \, \left( \xi + (q-2) \, \frac{(z : \xi) \, z}{|z|^2} \right)
\]
and
\[
\big( \partial A_q(z) : \xi \big) : \xi 
= |z|^{q-2} \, \left( |\xi|^2 + (q-2) \, \frac{(z : \xi)^2}{|z|^2} \right).
\]

Regarding second-order tensors, the following inequality is well-known (cf.\ \cite[(4.51)]{KM}).

\begin{lemma} \label{tensor ineq}
Let $q \in (1,\infty)$.
There exists a $c=c(n;N;p) \le 1$ such that
\[
\Big(|z_2|^{q-1}z_2-|z_1|^{q-1}z_1 \Big):(z_2-z_1)\geq c \, \Big(|z_2|^2+|z_1|^2\Big)^{(q-2)/2} \, |z_2-z_1|^2
\]
for all second-order tensors $z_1$ and $z_2$.
\end{lemma}

We end this section with the definition of a $q$-harmonic map.

\begin{definition}
Let $q \in (1,\infty)$.
A function $v \in W^{1,q}(\R^n)$ is said to be $q$-harmonic if 
\[
\int_{\R^n} |Dv|^{q-2} \, Dv : D\va \, dx = 0
\]
for all $\va \in C_c^\infty(\R^n)$.
\end{definition}

\section{A comparison estimate} \label{compare}

In this section we prove a comparison estimate between the weak solutions of \eqref{5hh070120148} and a $p$-harmonic map, which is the content of Proposition \ref{prop4.1}.

In what follows it is convenient to denote
\begin{equation}\label{4.2}
q_0=\frac{\beta' \, (p-1) \, n}{\beta' \, (n-1)-n}.
\end{equation}
Note that $q_0 \in (1,p)$.
Also set $B_\sigma = B_\sigma(0)$ for all $\sigma \in (0,1]$.

\begin{proposition} \label{prop4.1}
Let $\varepsilon>0$, $M\geq 1$ and $\beta \in (1,\infty)$ be such that $\frac{np}{n-p}<\beta'<\frac{n}{n(2-p)-1}$.
Let $1< q <q_0$ and $B = B_r(x_0)$ be a ball in $\R^n$.
Suppose $u\in W^{1,p}(B)$ satisfy
\begin{equation}\label{4.1}
\tb_{B}|u|dx \leq Mr.
\end{equation}  
Then there exists a positive constant $\delta=\delta\left(n,N,p,q,M,\varepsilon\right)\in (0,1)$ such that if
\begin{equation}\label{4.3}
\left|\tb_{B}|D u|^{p-2}D u: D \varphi \, dx\right|\leq \frac{\delta}{r}\left(\tb_{B} |\varphi(x)|^{\beta'} \, dx\right)^{1/\beta'}
\end{equation}
for all $\varphi\in W^{1,p}_0(B)\cap L^{\beta'}(B)$, then there exist a constant $c = c(n,N,p,q) > 0$ and a $p$-harmonic map $v\in W^{1,p}(\frac{1}{2}B)$ such that 
\[
\Big(\tb_{\frac{1}{2}B}\left|Du-Dv\right|^q \, dx\Big)^{1/q} \leq \varepsilon
\]
as well as
\[
\tb_{\frac{1}{2}B}|v| \, dx\leq M \, 2^n r\quad \text{and}\quad \Big(\tb_{\frac{1}{2}B}|Dv|^q \, dx\Big)^{1/q}\leq cM.
\]
\end{proposition}


We divide the proof of Proposition \ref{prop4.1} into several parts.
To begin with, recall the following self-improving property of reverse Holder inequalities (cf.\ \cite[Lemma 3.38]{HK}).

\begin{lemma} \label{lem3.1} 
Let $0<q<a<\gamma<\vc$, $\xi \ge 0$ and $M\geq 0$.
Let $\nu$ be a non-negative Borel measure with finite total mass and $B \subset \R^n$ be a ball.
Suppose $0 \le g \in L^p(U,v)$ satisfies the following: there exists a $c_0 > 0$ such that
\[
\left(\tp_{\sigma_1 B}g^\gamma d\nu\right)^{1/\gamma} \leq \frac{c_0}{(\sigma -\sigma_1)^{\xi}} \left(\tp_{\sigma B}g^a d\nu \right)^{1/a} +M
\]
for all $\kappa \leq \sigma_1 < \sigma \leq 1,$ where $\kappa \in (0,1).$ Then there exists a $c = c(c_0,\xi,\sigma,a,q) > 0$ such that
\[
\left(\tp_{\sigma_1 B}g^\gamma \, d\nu\right)^{1/\gamma} \leq \frac{c}{(1 -\sigma)^{\zeta}} \left[\left(\tp_{\sigma B}g^q \, d\nu\right)^{1/q} +M\right]
\]
for all $\sigma \in (\kappa,1),$ where
\[
\zeta :=\frac{\xi \, p \,(\gamma -q)}{q \, (\gamma -a)}.
\]
\end{lemma}

Next we will establish suitable a priori estimates for (scaled) weak solutions of \eqref{5hh070120148} under the assumptions in Proposition \ref{prop4.1}.

\begin{lemma} \label{thn11}
Let $M$ and $\beta$ be as in Proposition \ref{prop4.1}.
Let $\delta \in (0,1)$.
Suppose $\overline{u} \in W^{1,p}(B_1)$ satisfies
\begin{equation}\label{4.6}
\tb_{B_1} |\ngang{u}|dx\leq 1
\end{equation}
and
\begin{equation}\label{4.7}
\left|\tb_{B_1} |D\ngang{u}|^{p-2}D\ngang{u}:D\eta dx\right|\leq M^{1-p}\delta \norm{\eta}_{L^{\beta'}(B_1)}
\end{equation}
for all $\eta \in W^{1,p}_0(B_1) \cap L^{\beta'}(B_1)$.
Then there exists a $c = c(n,N,p,q)$ such that
\[
\|\overline{u}\|_{W^{1,q}(B_{3/4})}
\le c
\]
for all $q \in (1,q_0)$.
\end{lemma}

\begin{proof}
The main idea is to test \eqref{4.7} with suitable test functions.
Following \cite[Proof of Theorem 4.1]{KM} consider for each $t > 0$ the truncation operator $T_t:\R^N\ax\R^N$ defined by
\begin{equation}\label{4.8}
T_t(z):=\min\left\{1,\frac{t}{|z|}\right\}z.
\end{equation}
By direct calculations, $DT_t:\R^N \ax \R^N \otimes \R^N$ is given by
\begin{equation}\label{4.9}
DT_t(z)=
\left\{
\begin{array}{ll}
I & \text{if} \,\,\,|z|\leq t \\
\frac{t}{|z|}\left(I-\frac{z \otimes z}{|z|^2}\right) & \text{if} \,\,\, |z|>t,
\end{array}
\right.
\end{equation}
where $I:\R^N \ax \R^N \otimes \R^N$ denotes the identity operator.

Now let $\phi \in \Cc(B_1;\R)$ be such that $0\leq \phi\leq 1$ and then choose
\[
\eta:=\phi^p T_t\left(\ngang{u}\right)
\]
as a test function in \eqref{4.7}. 
We have
\begin{align*}
D \eta &=\1_{\{\ngang{u}< t \}}\left(\phi^p D\ngang{u}+p\phi^{p-1}\ngang{u}\otimes D\phi\right)\\
&\,\,\,\,+\1_{\{\ngang{u}\geq t \}} \frac{t}{|\ngang{u}|}\left(\phi^p(I-P)D\big(\ngang u\big)+p\phi^{p-1}\ngang{u}\otimes D\phi\right),
\end{align*}
where $P:=\frac{\ngang{u}\otimes \ngang{u}}{|\ngang{u}|^2}$.
Also notice that
\begin{equation}\label{4.11}
D\ngang{u}:\left[(I-P)D\ngang{u}\right]=|D\ngang{u}|^2-\frac{{u}^\alpha D_j\ngang{u}^\alpha {u}^k D_j\ngang{u}^k }{|u|^2}=|D\ngang{u}|^2-\frac{\sum_{j=1}^n (D_ju \cdot u)^2}{|u|^2}\geq 0
\end{equation}
and 
\[
\begin{aligned}
\norm{\eta}_{L^{\beta'}(B_1)}
& = \left(\int_{B_1 }\left|T_t\left(\ngang{u}\right)\right|^{\beta'} \phi^{p{\beta'}}dx\right)^{1/{\beta'}}=\left(\int_{B_1}\left|T_t\left(\ngang{u}\right)\right|^{\theta {\beta'}} \left|T_t\left(\ngang{u}\right)\right|^{{\beta'} (1-\theta)} \phi^{p{\beta'}}dx\right)^{1/{\beta'}}\\
& \leq t^\theta	\norm{\ngang{u}\phi^{\frac{p}{1-\theta}}}^{1-\theta}_{L^{{\beta'}(1-\theta)}(B_1)},
\end{aligned}	
\]
where $0<\theta<1$.

Substituting these into \eqref{4.7} and using Young's inequality we obtain
\begin{eqnarray}
\int_{B_1 \cap \left\{|\ngang{u}|<t\right\}}|D \ngang{u}|^p\phi^pdx 
&\leq& c\int_{B_1 \cap \left\{|\ngang{u}|<t\right\}}|\ngang{u}|^p|D\phi|^pdx +cM^{1-p}\delta  t^\theta	\norm{\ngang{u}\phi^{\frac{p}{1-\theta}}}^{1-\theta}_{L^{{\beta'}(1-\theta)}(B_1)}
\nonumber
\\
&& {}+ct\int_{B_1 \cap \left\{|\ngang{u}|\geq t\right\}}|D\ngang{u}|^{p-1}|D\phi|\phi^{p-1}dx
\label{4.12}
\end{eqnarray}
for some $c=c(n,N,p) > 0$.

For the rest of the proof we use $c = c(n,N,p)$ whose value may vary from line to line.

Next let $\gamma \in (0,1)$.
Multiplying \eqref{4.12} by $(1+t)^{-1-\gamma-\theta}$ and then integrating on $(0,\infty)$ with respect to $t$ give
\[
\begin{aligned}
\frac{1}{\theta+\gamma}\int_{B_1} \frac{|D\ngang{u}|^p\phi^p}{(1+|\un|)^{\gamma+\theta}}dx
&\leq \frac{c}{\gamma+\theta}\int_{B_1}(1+|\un|)^{p-\gamma-\theta}|D\phi|^pdx\\
&\quad  +\frac{c}{\gamma}\delta\norm{\ngang{u}\phi^{\frac{p}{1-\theta}}}^{1-\theta}_{L^{\beta'(1-\theta)}(B_1)}+c \int_{B_1} \frac{|\un||D\un|^{p-1}|D \phi|\phi^{p-1}}{(1+|\un|)^{\gamma+\theta}}dx.
\end{aligned}
\]

It follows from Young's inequality that
\[
\int_{B_1} \frac{|\un||D\un|^{p-1}|D \phi|\phi^{p-1}}{(1+|\un|)^{\gamma+\theta}}dx\leq \frac{1}{2c(\gamma+\theta)}\int_{B_1} \frac{|D\un|^p\phi^p}{(1+|\un|)^{\theta+\gamma}}dx+c(\gamma+\theta)^{p-1}\int_{B_1}(1+|\un|)^{-(\gamma+\theta)}|D\phi|^p|\un|^pdx.
\]

Consequently
\begin{equation}\label{4.14}
\int_{B_1} \frac{|D\un|^p\phi^p}{(1+|\un|)^{\theta+\gamma}}dx\leq c  \int_{B_1}(1+|\un|)^{p-\gamma-\theta}|D\phi|^pdx+\frac{c}{\gamma}\norm{\ngang{u}\phi^{\frac{p}{1-\theta}}}^{1-\theta}_{L^{\beta'(1-\theta)}(B_1)}
\end{equation}
The pointwise inequality $|D|\un||\leq |D\un|$ implies
\[
|D( (1+|\un|)^{1-\frac{\theta+\gamma}{p}}\phi)|^p \leq \frac{c|D\un|^p}{(1+|\un|)^{1+\gamma}}\phi^p+c(1+|\un|)^{p-\theta-\gamma}|D\phi|^p.
\]
Combining with \eqref{4.14}, we obtain
\begin{equation}\label{Z1}
\begin{aligned}
\int_{B_1}	|D( (1+|\un|)^{1-\frac{\theta+\gamma}{p}}\phi)|^pdx 
& \leq c \int_{B_1}(1+|\un|)^{p-\gamma-\theta}|D\phi|^pdx
+\frac{c}{\gamma}\norm{\ngang{u}\phi^{\frac{p}{1-\theta}}}^{1-\theta}_{L^{\beta'(1-\theta)}(B_1)}.
\end{aligned}
\end{equation}

Applying Sobolev's inequality to \eqref{Z1} and combining the derived estimate with \eqref{4.14} yield
\begin{equation}\label{2.18}
\begin{aligned}
\int_{B_1} \frac{|D\un|^p\phi^p}{(1+|\un|)^{\theta+\gamma}} \, dx+	\left(	\int_{B_1}	 (1+|\un|)^{\frac{(p-\theta-\gamma)n}{n-p}} \, \phi^{\frac{pn}{n-p}} \, dx\right)^{\frac{n-p}{n}}
\leq c & \int_{B_1}(1+|\un|)^{p-\gamma-\theta} \, |D\phi|^p \, dx 
\\
&{} +\frac{c}{\gamma} \, \norm{\ngang{u} \, \phi^{\frac{p}{1-\theta}}}^{1-\theta}_{L^{\beta'(1-\theta)}(B_1)}.
\end{aligned}
\end{equation}

Next let $7/8 \leq \sigma_1<\sigma\leq 1$ and $\psi\in \Cc(B_\sigma)$ be such that 
\[
0\leq \psi \leq 1,
\quad
\psi|_{B_{\sigma_1}} = 1
\quad
\mbox{and}
\quad
|D\psi|\leq \frac{100}{\sigma-\sigma_1}.
\]
With this choice of test function, we deduce from \eqref{2.18} that
\begin{equation}\label{thn5}
\begin{aligned}
\left(	\int_{B_{\sigma_1}}	 (1+|\un|)^{\frac{(p-\theta-\gamma)n}{n-p}} \, dx\right)^{\frac{n-p}{n}}
&\leq  \frac{c}{\sigma -\sigma_1}\int_{B_{\sigma}}(1+|\un|)^{p-\theta-\gamma} \, dx +\frac{c}{\gamma}\norm{\ngang{u}}^{1-\theta}_{L^{\beta'(1-\theta)}(B_\sigma)}
\end{aligned}
\end{equation}	
for all $\gamma,\theta \in (0,1).$

Now we choose $\theta, \gamma \in (0,1)$ such that $p - \theta - \gamma \ge 1$.
Then thanks to Lemma \ref{lem3.1} and \eqref{4.6}, we get
\begin{eqnarray}
\left(	\int_{B_{\sigma_1}}	 (1+|\un|)^{\frac{(p-\theta-\gamma)n}{n-p}} \, dx\right)^{\frac{n-p}{n}}
&\leq  \frac{c}{1-\sigma} 
+ \frac{c}{\gamma} \, \norm{\ngang{u}}^{1-\theta}_{L^{\beta'(1-\theta)}(B_\sigma)}
\nonumber
\le \frac{c}{1-\sigma} 
+ \frac{c}{\gamma} \, \norm{1+\ngang{u}}^{1-\theta}_{L^{\beta'(1-\theta)}(B_\sigma)}.
\label{thn7}
\end{eqnarray}	

The lemma can now be achieved by iterating \eqref{thn7} multiple times.
Indeed if we denote $b=\frac{n}{n-p}$ then \eqref{thn7} reads 
\begin{equation}\label{thn6}
\begin{aligned}
\norm{1+|\un|}^{p-\theta-\gamma}_{L^{b(p-\theta-\gamma)}(B_{\sigma_1})}
&\leq  \frac{c}{1 -\sigma}  + \frac{c}{\gamma}\norm{1+|\un|}^{1-\theta}_{L^{\beta'(1-\theta)}(B_\sigma)}.
\end{aligned}
\end{equation}

For each $k \in \N^*$ set $\gamma_k=(2\beta')^{-k}$ and $\theta_k$ such that
\[
\begin{cases}
\theta_1 =1-\frac{1}{\beta'},\\
\theta_{k+1}=1-\frac{b}{\beta'}(p-\theta_k-\gamma_k) \in (0,1).
\end{cases}
\]

Using \eqref{thn6}, \eqref{thn11}, we obtain
\begin{equation}\label{thn8}
\norm{1+|\un|}_{L^{\beta'(1-\theta_k)}(B_{7/8})}+\norm{1+|\un|}_{L^{b(p-\theta_k-\gamma_k)}(B_{7/8})} \leq  c_k
\end{equation}
for all $k\in \N^*$, where $ c_k=c_k(n,N,p,k).$

By extracting a subsequence when necessary, we may assume without loss of generality that $\li{k\to \vc}\theta_k=\theta_0$.  
Then
\[
\beta' \, (1-\theta_0)=\frac{(p-\theta_0) \, n}{n-p}
\]
or equivalently 
\[
\theta_0=\frac{\beta' \, (n-p)-p \, n}{\beta' \, (n-p)-n}.
\]

Observe that for all $a_1>0$ there exists a $k_1 \in \N^*$ such that $\theta_0 +\frac{a_1}{b} \geq \theta_{k_1}+\gamma_{k_1}$. 
Therefore \eqref{thn8} implies
\begin{equation}\label{thn4}
\int_{B_{7/8}}	 (1+|\un|)^{\frac{(p-\theta_0)n}{n-p}-a_1}dx \leq c(n,N,p,a_1)
\end{equation} 
for all $a_1 > 0$.
Choosing a suitable test function in \eqref{4.14} leads to
\[
\int_{B_{3/4}} \frac{|D\un|^p}{(1+|\un|)^{\theta+\gamma}}dx\leq c\int_{B_{7/8}}(1+|\un|)^{p-\gamma-\theta}dx+\frac{c}{\gamma} \norm{\ngang{u}}^{1-\theta}_{L^{\beta'(1-\theta)}(B_{7/8})}.
\]
Then \eqref{thn8} in turn implies
\begin{equation}\label{thn10}
\int_{B_{3/4}} \frac{|D\un|^p}{(1+|\un|)^{\theta_k+\gamma_k}}dx\leq  c_k,
\end{equation}
for all $k \in \N^*$, where $c_k=c_k(n,N,p,k).$

Analogously for all $a_2>0$ there exists a $k_2 \in \N^*$ such that $\theta_0 +a_2 > \theta_{k_2}+\gamma_{k_2}$.
Therefore \eqref{thn10} gives
\begin{equation}\label{thn12}
\int_{B_{3/4}} \frac{|D\un|^p}{(1+|\un|)^{\theta_0+a_2}} \, dx \leq c(n,N,p,a_2)
\end{equation}	
for all $a_2>0$.

Now let $a=\frac{\beta'(p-1)n}{\beta'(n-1)-n}$ and apply Holder's inequality for the exponent $\frac{p}{a-a_2}$ to arrive at
\begin{equation}\label{4.23}
\begin{aligned}
\int_{B_{3/4}} |D\un|^{\frac{\beta'(p-1)n}{\beta'(n-1)-n}-a_2} \, dx& =\int_{B_{3/4}}	 |D\un|^{a-a_2} (1+|\un|)^{-(\theta_0+a_2)(a-a_2)/p} \, (1+|\un|)^{(\theta_0+a_2)(a-a_2)/p} \, dx
\\
\leq & \left(\int_{B_{3/4}} \frac{|D\un|^p}{(1+|\un|)^{\theta_0+a_2}} \, dx\right)^{(a-a_2)/p}
\\
& \times \left(\int_{B_{3/4}}(1+|\un|)^{\frac{(\theta_0+a_2)(a-a_2)}{p-a+a_2}} \, dx \right)^{(p-a+a_2)/p}. 
\end{aligned}
\end{equation}

Since $(a-a_2)/(p-a+a_2) < a/(p-a)$ and $\beta'>np/(n-p)$, one has
\[
\frac{\theta_0a}{p-a}<\frac{(p-\theta_0)n}{n-p}
\]
and so
\[
\frac{(\theta_0+a_2)(a-a_2)}{p-a+a_2}<  \frac{(\theta_0+a_2)a}{p-a}<\frac{(p-\theta_0)n}{n-p}-a_1
\]
for all $a_1$, $a_2>0$ small enough.

By putting \eqref{thn4}, \eqref{thn12} and \eqref{4.23} together, 
\begin{equation}\label{thn13}
\int_{B_{3/4}} |D\un|^{\frac{\beta'(p-1)n}{\beta'(n-1)-n}-a_2}dx\le c(n,N,p,a_2)
\end{equation}
for sufficiently small $a_2>0$.

We now combine \eqref{thn5} and \eqref{thn13} to conclude that
\[
\int_{B_{3/4}}	 |\un|^{\frac{\beta'(p-1)n}{\beta'(n-p)-n}-a_1} \, dx
\leq c(n,N,p,a_1)
\quad \mbox{and} \quad
\int_{B_{3/4}} |D\un|^{\frac{\beta'(p-1)n}{\beta'(n-1)-n}-a_2} \, dx
\leq c(n,N,p,a_2)
\]
for all sufficiently small $a_1$, $a_2>0$ (and so trivially for all larger values of $a_1$ and $a_2$).

This verifies our claim.
\end{proof}

\begin{lemma} \label{u tilde}
Let $M$ and $\beta$ be as in Proposition \ref{prop4.1}.
Let $\{u_j\}_{j \in \N^*} \subset W^{1,p}(B_1)$ satisfy
\begin{equation} \label{4.27}
\tb_{B_1}|u_j|dx \leq 1
\end{equation}
and
\begin{equation}\label{4.28}
\left|\tb_{B_1}|D u_j|^{p-2}D u_j: D \varphi \, dx\right|\leq M^{1-p} \, 2^{-j}\left(\tb_{B_1} |\varphi(x)|^{\beta'} \, dx\right)^{1/\beta'}
\end{equation}
for all $\varphi\in W^{1,p}_0(B_1)\cap L^{\beta'}(B_1)$.
Then there exists a $\tilde{u} \in W^{1,q}(B_{3/4})$ such that
\[
\lim_{j \to \infty} u_j = \tilde{u} 
\quad \mbox{in } W^{1,q}(B_{3/4})
\]
for all $q \in (1,q_0)$.
Moreover,
\begin{equation} \label{u tilde eqn}
\tb_{B_{1/2}} |D \tilde{u}|^{p-2} \, D \tilde{u} : D\va \, dx = 0
\end{equation}
for all $\va \in C_c^\infty(B_{1/2})$.
\end{lemma}

\begin{proof}
Let $1<q<q_0$ and $q_1=(q+q_0)/2$.
By Lemma \ref{thn11}, there exists a $c = c(n,N,p,q)$ such that
\begin{equation}\label{4.31}
\int_{B_{3/4}}|D\un_j|^qdx\le c
\quad \text{and} \quad 
\int_{B_{3/4}} |D\un_j|^{q_1}dx\le c
\end{equation}
uniformly in $j \in \N^*$. 

For convenience we will constantly use $c = c(n,N,p,q)$ without mentioning further, the value of which may vary from line to line.

By passing to a subsequence if necessary, we may assume there exist $\nga{u} \in W^{1,q}(B_{3/4})$, $b \in L^{q/(p-1)}(B_{3/4})$ and $h \in L^q(B_{3/4})$ such that
\begin{eqnarray}
&& \tp_{B_{3/4}}|D\nga{u}|^qdx+\su_{j}\tp_{B_{3/4}}|D \un_j|^qdx+\su_j\tp_{B_{3/4}}|D\un_j|^{q_1}dx<\vc,
\label{4.32a}
\\
&& D\un_j \hty D\nga{u}, \quad |D\un_j-D\nga{u}|\hty h \quad \text{weakly in}\,\, L^q(B_{3/4}),
\label{4.32b}
\\
&& |D\un_j|^{p-2}D\un_j \hty b \quad \text{weakly in}\quad L^{q/(p-1)}(B_{3/4}) \quad \mbox{and}
\label{4.32c}
\\
&& \un_j \hoi \nga{u} \quad \text{strongly in} \, L^q(B_{3/4})\,\, \text{and pointwise in}\, B_{3/4}.
\label{4.32d}
\end{eqnarray} 
As a consequence of \eqref{4.27} and \eqref{4.31} we have
\begin{equation} \label{4.33}
\tb_{B_{3/4}}|\nga{u}|dx\leq 2^n
\quad \mbox{and} \quad
\int_{B_{3/4}} |D\nga{u}|^qdx \le c.
\end{equation}

Next we aim to prove that $h=0$ almost everywhere, from which the lemma follows at once.
To this end it suffices to show that
\begin{equation}\label{4.38}
h\left(\ngang{x}\right)=0
\end{equation}
for all $\ngang{x}\in B_{3/4}$ which is a Lebesgue point simultaneously for $\tilde{u}$, $D\tilde{u}$, $h$ and $b$, that is,
\begin{equation}\label{4.36}
\mathop {\lim\limits}_{\theta \to 0}\tb_{B_\theta\left(\ngang{x}\right)}\Bigg[|\nga{u}-\nga{u}\left(\ngang{x}\right)|+|D\nga{u}-D\nga{u}\left(\ngang{x}\right)|+|h-h\left(\ngang{x}\right)|+|b-b\left(\ngang{x}\right)|^{1/(p-1)}\Bigg]^qdx=0
\end{equation}
and
\begin{equation}\label{4.37}
|\nga{u}\left(\ngang{x}\right)|+|D\nga{u}\left(\ngang{x}\right)|+|h\left(\ngang{x}\right)|+|b\left(\ngang{x}\right)|<\vc.
\end{equation}

To see this, with \eqref{4.38} in mind, $D\ngang{u}_j\to D\nga{u}$ strongly in $L^1(B_{3/4})$. 
Whence the second bound in \eqref{4.33} and interpolation yield
\[
\norm{D\ngang{u}_j-D\nga{u}}_{L^q(B_{3/4})}\leq \norm{D\ngang{u}_j-D\nga{u}}_{L^1(B_{3/4})}\norm{D\ngang{u}_j-D\nga{u}}_{L^{q_1}(B_{3/4})}^{1-\theta} \htu{j \to \vc} 0,
\]
where $\theta$ is such that $1/q=\theta+(1-\theta)/{q_1}$.

Now back to the proof of \eqref{4.38}, let $\ngang{x}\in B_{3/4}$ be a simultaneous Lebesgue point for $\tilde{u}$, $D\tilde{u}$, $h$ and $b$.
Set
\[
\alpha_\sigma(x):=\big(\nga{u}\big)_{B_\sigma\left(\ngang{x}\right)}+ D\nga{u}\left(\ngang{x}\right) \cdot (x-\ngang{x})
\]
for all $\sigma \in (0,3/4)$.
Poincare's inequality for $\alpha_\sigma$ implies
\begin{equation}\label{4.40}
\lim\limits_{\sigma\to 0} \tb_{B_\sigma\left(\ngang{x}\right)} \left|\frac{\nga{u}-\alpha_\sigma}{\sigma}\right|^q dx \leq c \lim\limits_{\sigma \to 0} \tb_{B_\sigma\left(\ngang{x}\right)}|D \nga u-D \nga u\left(\ngang{x}\right)|^q dx=0.
\end{equation}

By \eqref{4.32b} we have
\begin{eqnarray}
h\left(\ngang{x}\right)
&=& \li{\sigma\to 0} \li{j \to \vc} \tb_{B_{\sigma/2}\left(\ngang{x}\right)} |D \un_j-D\nga{u}|dx
\nonumber 
\\
&=& \li{\sigma\to 0} \li{j \to \vc} \tb_{B_{\sigma/2}\left(\ngang{x}\right)} \1_{\{|\un_j-\alpha_\sigma|<
\sigma\}}|D \un_j-D\nga{u}|dx
+ \li{\sigma\to 0} \li{j \to \vc} \tb_{B_{\sigma/2}\left(\ngang{x}\right)} \1_{\{|\un_j-\alpha_\sigma|\geq
\sigma\}}|D \un_j-D\nga{u}|dx	
\nonumber
\\
&=:& I + II.
\label{4.41} 
\end{eqnarray}
We aim to show that $I = II = 0$.
For this we estimate each term separately.
Term $II$ turns out to be easier to estimate so we do it first.

{\bf Term $II$}:
We first show that
\begin{equation} \label{t2 aim}
\lim_{j \to \infty} \tb_{B_{\sigma/2}\left(\ngang{x}\right)} \1_{\{|\un_j-\alpha_\sigma|\geq
\sigma\}}|D \un_j-D\nga{u}| \, dx
\le \tb_{B_{\sigma/2}\left(\ngang{x}\right)} \1_{\{|\nga{u}-\alpha_\sigma|\geq
\sigma\}}h \, dx.
\end{equation}

To this end note that
\begin{equation*}
\begin{aligned}
\tb_{B_{\sigma/2}\left(\ngang{x}\right)} \1_{\{|\un_j-\alpha_\sigma|\geq
\sigma\}}|D \un_j-D\nga{u}|dx
\leq & \tb_{B_{\sigma/2}\left(\ngang{x}\right)} \1_{\{|\un_j-\nga{u}|\geq
\sigma/2\}}|D \un_j-D\nga{u}|dx
\\
& +\tb_{B_{\sigma/2}\left(\ngang{x}\right)} \1_{\{|\nga{u}-\alpha_\sigma|\geq
\sigma/2\}}|D \un_j-D\nga{u}|dx.
\end{aligned}
\end{equation*}
By invoking \eqref{4.32a} and \eqref{4.32d} one has
\begin{align*}
\tb_{B_{\sigma/2}\left(\ngang{x}\right)} \1_{\{|\un_j-\nga{u}|\geq \sigma\}}|D \un_j-D\nga{u}|dx 
\leq & \, \Bigg(\tb_{B_{\sigma/2}\left(\ngang{x}\right)} |D \un_j-D\nga{u}|^qdx\Bigg)^{1/q} 
\Bigg(\frac{|\{x\in B_{3/4}:|\un_j-\nga{u}|\}|\geq \sigma/2}{|B_{\sigma/2}\left(\ngang{x}\right)|}\Bigg)^{1/q'} 
\\
& \htu{j \to \vc} 0.
\end{align*}
This justifies \eqref{t2 aim}.

Next we use \eqref{4.36}, \eqref{4.37} and \eqref{4.40} to obtain
\begin{align*}
\tb_{B_{\sigma/2}\left(\ngang{x}\right)} \1_{\{|\nga{u}-\alpha_\sigma| 
\geq \sigma/2\}}h \, dx
&\leq \Bigg(\tb_{B_{\sigma}\left(\ngang{x}\right)}h^qdx\Bigg)^{1/q}\Bigg(\tb_{B_{\sigma}\left(\ngang{x}\right)}\1_{\{|\nga{u}-\alpha_\sigma|\geq \sigma/2\}} \, dx\Bigg)^{1/q'}\\
& \leq c\Bigg[ \Bigg(\tb_{B_{\sigma}\left(\ngang{x}\right)}|h-h\left(\ngang{x}\right)|^q \, dx\Bigg)^{1/q}
+h\left(\ngang{x}\right)\Bigg]\Bigg(\tb_{B_\sigma\left(\ngang{x}\right)} \left|\frac{\ngang{u}-\alpha_\sigma}{\sigma}\right|^q dx\Bigg)^{1/q'}
\\
& \quad \htu{\sigma \to 0} 0.
\end{align*}
Hence $II = 0$.

{\bf Term $I$}:
One has
\begin{align*}
\tb_{B_{\sigma/2}\left(\ngang{x}\right)} \1_{\{|\un_j-\alpha_\sigma|<
\sigma\}}|D \un_j-D\nga{u}|dx 
\leq & \, \tb_{B_{\sigma/2}\left(\ngang{x}\right)} \1_{\{|\un_j-\alpha_\sigma|<
\sigma\}}|D \un_j-D\alpha_\sigma|dx
\\
& +2^n\tb_{B_{\sigma/2}\left(\ngang{x}\right)} \1_{\{|\un_j-\alpha_\sigma|<
\sigma\}}|D\nga{u}-D\alpha_\sigma|dx.
\end{align*}
Since 
\[
\li{\sigma\to 0} \lsu{j \to \vc} \tb_{B_{\sigma/2}\left(\ngang{x}\right)} \1_{\{|\un_j-\alpha_\sigma|< 
\sigma\}}|D \nga{u}-D\alpha_\sigma|dx\leq  \li{\sigma\to 0} \tb_{B_{\sigma/2}\left(\ngang{x}\right)} |D \nga{u}-D\nga{u}\left(\ngang{x}\right)|dx=0
\]
by \eqref{4.36}, it remains to show that
\begin{equation}\label{4.45}
\li{\sigma\to 0} \lsu{j \to \vc} \tb_{B_{\sigma/2}\left(\ngang{x}\right)} \1_{\{|\un_j-\alpha_\sigma|< 
\sigma\}}|D\un_j-D\alpha_\sigma|dx=0.
\end{equation}

By Holder's inequality,
\begin{eqnarray*}
\tb_{B_{\sigma/2}\left(\ngang x\right)} \1_{\{|\un_j-\alpha_\sigma|<\sigma\}}|D\un_j-D\alpha_\sigma| dx 
&\leq& \Bigg(\tb_{B_{\sigma/2}\left(\ngang x\right)}  \1_{\{|\un_j-\alpha_\sigma|<\sigma\}}\Big(|D\un_j|+|D\alpha_\sigma|\Big)^{p-2}|D\un_j-D\alpha_\sigma|^2  dx\Bigg)^{1/2}\\
&& {} \times \Bigg(\tb_{B_{\sigma/2}\left(\ngang x\right)}  \Big(|D\un_j|+|D\alpha_\sigma|\Big)^{2-p}  dx\Bigg)^{1/2}.
\end{eqnarray*}
The second integral on the right-hand side is bounded uniformly in $j$ due to \eqref{4.31} and \eqref{4.33}.
Hence to achieve \eqref{4.45}, it suffices to show that 
\[
\li{\sigma\to 0} \lsu{j \to \vc}
\tb_{B_{\sigma/2}\left(\ngang x\right)}  \1_{\{|\un_j-\alpha_\sigma|<\sigma\}}\Big(|D\un_j|+|D\alpha_\sigma|\Big)^{p-2}|D\un_j-D\alpha_\sigma|^2  dx
= 0.
\]

To this end, let $\phi \in \Cc(B_\sigma(\overline{x}))$ be such that
\[
0\leq \phi \leq 1,
\quad 
\phi|_{B_{\sigma/2}\left(\xn\right)} = 1 
\quad \mbox{and} \quad 
|D\phi|\leq \frac{4}{\sigma}.
\]
Set $\eta:=\phi \, T_{\sigma}(\un_j-\alpha_\sigma)$, where $T_\sigma$ is defined by \eqref{4.8}.
It follows from \eqref{4.9} that
\begin{align*}
& \left(|D\un_j|^{p-2}D\un_j-|D\alpha_\sigma|^{p-2}D\alpha_\sigma \right):D\eta\\
&\quad = \, \1_{\{\un_j-\alpha_\sigma\}} \left[\left(|D\un_j|^{p-2}D\un_j-|D\alpha_\sigma|^{p-2}D\alpha_\sigma
\right) :D(\un_j-\alpha_\sigma)\right]\phi  \\
& \quad \quad {}+ \1_{\{|\un_j-\alpha_\sigma|>\sigma\}} \, \frac{\sigma }{|\un_j-\alpha_\sigma|}\left[\left(|D\un_j|^{p-2}D\un_j-|D\alpha_\sigma|^{p-2}D\alpha_\sigma
\right) :(I-P_j)D(\un_j-\alpha_\sigma)\right]\phi\\
& \quad \quad {} +\left(|D\un_j|^{p-2}D\un_j-|D\alpha_\sigma|^{p-2}D\alpha_\sigma
\right) : \left[T_\sigma (\un_j-\alpha_\sigma)\otimes D\phi\right]\\
&\quad =:\,G^1_{j,\sigma}(x)+G^2_{j,\sigma}(x)+G^3_{j,\sigma}(x),
\end{align*}
where
\[
P_j:=\frac{\left(\un_j-\alpha_\sigma\right)\otimes \left(\un_j-\alpha_\sigma\right)}{|\un_j-\alpha_\sigma|^2}
\quad \mbox{and} \quad
P:=\frac{\left(\nga u-\alpha_\sigma\right)\otimes \left(\nga u-\alpha_\sigma\right)}{|\nga u-\alpha_\sigma|^2}.
\]
Since $\alpha_\sigma$ is affine, one has
\[
\int_{B_1} |D\alpha_\sigma|^{p-2}D\alpha_\sigma:D\eta \, dx=0.
\]
Therefore
\begin{equation}\label{4.46}
0\leq \tb_{B_\sigma\left(\ngang{x}\right)} G^1_{j,\sigma}(x) \, dx
\leq 2^{-j}\sigma^{1-n}-\tb_{B_\sigma\left(\ngang{x}\right)} G^2_{j,\sigma}(x) \, dx-\tb_{B_\sigma\left(\ngang{x}\right)} G^3_{j,\sigma}(x) \, dx,
\end{equation}
where we used the monotonicity of the vector field $z\ax |z|^{p-2}z$ in the first step. 

Next we estimate the two integrals on
the right-hand side of the above inequality.

{\bf Integral of $G^3_{j,\sigma}$}: 
First we deduce from \eqref{4.31} that $\{|D\un_j|^{p-2}D\un_j\}_{j \in \N^*}$ is bounded in $L^{q/(p-1)}$.
This together with \eqref{4.32c} and \eqref{4.32d} imply that
\[
\li{j \to \vc} \tb_{B_\sigma\left(\ngang{x}\right)} G^3_{j,\sigma}(x) \, dx
= \tb_{B_\sigma\left(\ngang{x}\right)}\left(b-|D\alpha_\sigma|^{p-2}D\alpha_\sigma\right):\left[T_\sigma(\ngang{u}-\alpha_\sigma)\otimes D\phi\right]dx.
\]
Holder's inequality then gives
\begin{eqnarray*}
&& \left|\tb_{B_\sigma\left(\ngang{x}\right)}\left(b-|D\alpha_\sigma|^{p-2}D\alpha_\sigma\right):\left[T_\sigma(\nga{u}-\alpha_\sigma)\otimes D\phi\right]dx\right|\\
&\le& c \left(\tb_{B_\sigma\left(\xn\right)} |b-b\left(\xn\right)|^{q/(p-1)}+|b\left(\xn\right)|^{q/(p-1)}+|D\nga{x}\left(\ngang{x}\right)|^q dx\right)^{\frac{p-1}{q}}\\
&& {} \times \left(\tb_{B_\sigma\left(\xn\right)} \left(\frac{\min\{\sigma,|\nga{u}-\alpha_\sigma|\}}{\sigma}\right)^{\frac{q}{q-(p-1)}}dx\right)^{1-\frac{p-1}{q}}.
\end{eqnarray*}
Note that the first integral on the right-hand side is bounded.
For the second integral, we have
\begin{align*}
\tb_{B_\sigma\left(\xn\right)} \left(\frac{\min\{\sigma,|\nga{u}-\alpha_\sigma|\}}{\sigma}\right)^{\frac{q}{q-(p-1)}}dx
&\leq \tb_{B_\sigma\left(\xn\right)} \left(\frac{\min\{\sigma,|\nga{u}-\alpha_\sigma|\}}{\sigma}\right)^qdx\\
&\leq \tb_{B_\sigma\left(\xn\right)} \left|\frac{\nga{u}-\alpha_\sigma}{\sigma}\right|^qdx \htu{\sigma \to 0}0,
\end{align*}
where we used the fact that $\frac{q}{q-(p-1)}>q$ and \eqref{4.40} in the first and second steps respectively.

Consequently
\[
\li{\sigma\to 0} \li{j \to \vc} \Bigg|\tb_{B_\sigma\left(\xn\right)}G^3_{j,\sigma}(x)dx\Bigg|=0.
\]

{\bf Integral of $G^2_{j,\sigma}$}: 
We have $D\un_j:\left[(I-P_j)D\un_j\right]\geq 0$ by a similar argument to that of \eqref{4.11}.
Therefore
\begin{eqnarray}
&& \left(|D\un_j|^{p-2}D\un_j-|D\alpha_\sigma|^{p-2}D\alpha_\sigma\right):(I-P_j)D(\un_j-\alpha_\sigma)
\nonumber
\\
&\ge& -|D\un_j|^{p-2} D\un_j:(I-P_j)D\alpha_\sigma-|D\alpha_\sigma|^{p-2}D\alpha_\sigma:(I-P_j)D(\un_j-\alpha_\sigma).
\label{4.48}
\end{eqnarray}

Observe also that $\1_{\{|\un_j-\alpha_\sigma|\geq \alpha_\sigma\}}P_j \to \1_{\{|\nga{u}-\alpha_\sigma|\geq \alpha_\sigma\}}P$ a.e.\ and hence strongly in $L^s(B_{3/4})$ for every $s \geq 1$.
The same also applies to the convergence $\1_{\{|\un_j-\alpha_\sigma|\geq \alpha_\sigma\}}|\un_j-\alpha_\sigma|^{-1} \to \1_{\{|\nga{u}-\alpha_\sigma|\geq \alpha_\sigma\}}|\nga{u}-\alpha_\sigma|^{-1}$. 
These in combination with \eqref{4.48} and \eqref{4.32c} yield that
\[
\begin{aligned}
\lsu{j \to \vc} \Bigg(-\tb_{B_\sigma\left(\xn\right)}G^2_{j,\sigma}(x)dx\Bigg)\leq& \tb_{B_\sigma\left(\xn\right)} b: (I-P)D\alpha_\sigma \frac{\sigma \1_{\{|\nga{u}-\alpha_\sigma|>\sigma\}}}{|\nga{u}-\alpha_\sigma|}dx\\
&{}+\tb_{B_\sigma\left(\xn\right)} |D\alpha_\sigma|^{p-2}D\alpha_\sigma:(I-P)D(\nga u-\alpha_\sigma)\frac{\sigma \1_{\{|\nga{u}-\alpha_\sigma|>\sigma\}}}{|\nga{u}-\alpha_\sigma|}dx.
\end{aligned}
\]

Next we estimate each on the right-hand side separately. 
As $q>p-1$ there exists an $s>1$ such that $\frac{q(s-1)}{q-p+1}\leq q$. 
Keeping in mind \eqref{4.40} one has
\begin{eqnarray*}
\left|\tb_{B_\sigma\left(\xn\right)} b: (I-P)D\alpha_\sigma \frac{\sigma \1_{\{|\nga{u}-\alpha_\sigma|>\sigma\}}}{|\nga{u}-\alpha_\sigma|}dx\right|
&\leq& c\tb_{B_\sigma\left(\xn\right)}|b|\left|\frac{\nga u -\alpha_\sigma}{\sigma}\right|^{s-1}dx\\
&\leq& c \Bigg(\tb_{B_\sigma\left(\xn\right)} |b|^{q/(p-1)}dx\Bigg)^{(p-1)/q}\Bigg(\tb_{B_\sigma\left(\xn\right)} \left|\frac{\un-\alpha_\sigma}{\sigma}\right|^qdx\Bigg)^{(s-1)/q}
\\
&& \htu{\sigma \to 0}0.
\end{eqnarray*}
At the same time,
\begin{eqnarray*}
&& \tb_{B_\sigma\left(\xn\right)} |D\alpha_\sigma|^{p-2}D\alpha_\sigma:(I-P)D(\nga u-\alpha_\sigma)\frac{\sigma \1_{\{|\nga{u}-\alpha_\sigma|>\sigma\}}}{|\nga{u}-\alpha_\sigma|}dx
\\
&\leq& c\tb_{B_\sigma\left(\xn\right)}|D(\nga u-\alpha_\sigma)|\left|\frac{\nga u -\alpha_\sigma}{\sigma}\right|^{q-1}dx\\
&\leq& c \Bigg(\tb_{B_\sigma\left(\xn\right)} |D\un-D\un\left(\ngang x\right)|^qdx\Bigg)^{1/q}\Bigg(\tb_{B_\sigma\left(\xn\right)} \left|\frac{\un-\alpha_\sigma}{\sigma}\right|^q dx\Bigg)^{1-1/q}
\\
&& \htu{\sigma \to 0}0.
\end{eqnarray*}
As a consequence,
\[
\lsu{\sigma \to 0} \lsu{j \to \vc}  \Bigg(-\tb_{B_\sigma\left(\xn\right)} G^2_{j,\sigma}(x) dx\Bigg)\leq 0.
\]
This finishes our estimate for the integral of $G^2_{j,\sigma}$.

Continuing with \eqref{4.46} we conclude that
\begin{equation}\label{4.50}
\lsu{\sigma \to 0} \lsu{j \to \vc}  \tb_{B_\sigma\left(\xn\right)} G^1_{j,\sigma}(x) dx =0.
\end{equation}

We proceed with the proof of \eqref{4.45}.
It follows from \eqref{4.50} and Lemma \ref{tensor ineq} that
\[
\lsu{\sigma \to 0} \lsu{j \to \vc} \tb_{B_\sigma\left(\xn\right)} \1_{\{|\un_j-\alpha_\sigma|<\sigma\}}\Big(|D\un_j|+|D\alpha_\sigma|\Big)^{p-2}|D\un_j-D\alpha_\sigma|^2 \phi \, dx = 0.
\]

Hence $I = 0$.

That $h(\overline{x})=0$ now follows from \eqref{4.41}, whence $Du \in L^q(B_{3/4})$.
Lastly, we let $j \longrightarrow \infty$ in \eqref{4.28} to obtain \eqref{u tilde eqn}.
This completes our proof.
\end{proof}

We now have enough preparation to derive Proposition \ref{prop4.1}.

\begin{proof}[{\bf Proof of Proposition \ref{prop4.1}}]
We proceed via a proof by contradiction.
Our arguments follow \cite[Step 5 in Proof of Theorem 4.1]{KM} closely.

For a contradiction, assume that there exist an $\epsilon > 0$ and sequences of balls $\{B_{r_j}(x_j)\}_{j \in \N^*}$ and $\{u_j\}_{j \in \N^*} \subset W^{1,p}(B_{r_j}(x_j))$ such that 
\begin{equation} \label{uj counter}
\tb_{B_{r_j}(x_j)} |u_j| \, dx \le M \, r_j
\quad \mbox{and} \quad
\left| \tb_{B_{r_j}(x_j)} |Du_j|^{p-2} \, Du_j : D\phi \, dx \right|
\le \frac{2^{-j}}{r_j} \, \|\phi\|_{L^{\beta'}(B_{r_j}(x_j))}
\end{equation}
for all $\phi \in W^{1,p}_0(B_{r_j}(x_j)) \cap L^{\beta'}(B_{r_j}(x_j))$, whereas
\[
\left( \tb_{B_{r_j/2}(x_j)} |Du_j - Dv|^q \right)^{1/q} > \epsilon
\]
for all $v \in W^{1,p}(B_{r_j/2}(x_j))$ being $p$-harmonic in $B_{r_j}(x_j)$ and satisfying
\[
\tb_{B_{r_j/2}(x_j)} |v| \, dx \le 2^n \, M \, r_j
\quad \mbox{and} \quad
\left( \tb_{B_{r_j/2}(x_j)} |Dv|^q \right)^{1/q} 
\le \left( \frac{2^n \, c}{|B_1} \right)^{1/q} \, M
\]
for all $q \in (1,q_0)$, where $c = c(n,N,p,q)$.

For the rest of the proof, $c$ will always denote a constant depending on $n$, $N$, $p$, $q$ only whose value may vary from line to line.

We first perform a scaling on $u_j$ for all $j \in \N$.
For convenience, we denote $u_0 = u$.
For each $j \in \N$ and  $\va \in W^{1,p}_0(B)\cap L^{\beta'}(B)$ let
\[
\overline{u}_j(x)=\frac{u_j(x_0+rx)}{Mr}
\quad \text{and}\quad 
\eta(x)=\frac{\va(x_0+rx)}{r}.
\] 
Then \eqref{4.1}, \eqref{4.3} and \eqref{uj counter} become
\begin{equation}\label{uj counter scale 1}
\tb_{B_1} |\ngang{u}_j|dx\leq 1
\end{equation}
and
\begin{equation}\label{uj counter scale 2}
\left|\tb_{B_1} |D\ngang{u}_j|^{p-2}D\ngang{u}_j:D\eta dx\right|\leq M^{1-p} \, \delta_j \, \norm{\eta}_{L^{\beta'}(B_1)},
\end{equation}
where 
\[
\delta_j := \left\{
\begin{array}{ll}
\delta & \mbox{if } j = 0,
\\
2^{-j} & \mbox{otherwise}.
\end{array}
\right.
\]

It follows from Lemma \ref{thn11} that
\[
\|\un_j\|_{W^{1,q}(B_{3/4})}
\leq c
\]
for all $q \in (1,q_0)$ and $j \in \N$.

Using Lemma \ref{u tilde} there exists a $\tilde{u} \in W^{1,q}(B_{3/4})$ such that
\[
\lim_{j \to \infty} u_j = \tilde{u} 
\quad \mbox{in } W^{1,q}(B_{3/4})
\]
for all $q \in (1,q_0)$ with the property that
\[
\tb_{B_{1/2}} |D \tilde{u}|^{p-2} \, D \tilde{u} : D\va \, dx = 0
\]
for all $\va \in C_c^\infty(B_{1/2})$.

We aim to show that $\nga u$ is $p$-harmonic. 
In particular,  we will show that $D\nga u \in L^p(B_{1/2})$.

Let $\phi \in \Cc(B_{3/4})$ be such that $0 \le \phi \le 1$ and $\phi|_{B_{1/2}} = 1$.
It follows from \eqref{4.12} that
\begin{eqnarray*}
\int_{B_1 \cap \left\{|\ngang{u}_j|<t\right\}}|D \ngang{u}_j|^p\phi^p \, dx 
&\leq& c\int_{B_1 \cap \left\{|\ngang{u}_j|<t\right\}}|\ngang{u}_j|^p|D\phi|^p \, dx +cM^{1-p} \, \delta_j \, t^\theta	\norm{\ngang{u}_j \phi^{\frac{p}{1-\theta}}}^{1-\theta}_{L^{{\beta'}(1-\theta)}(B_1)}
\nonumber
\\
&& {}+ct\int_{B_1 \cap \left\{|\ngang{u}_j|\geq t\right\}}|D\ngang{u}|^{p-1}|D\phi| \, \phi^{p-1} \, dx.
\end{eqnarray*}
By taking the inferior limit both sides of this inequality when $j \to \vc$ and then referring to Fatou's lemma for the left-hand side, one has
\[
\int_{B_{3/4}\cap \{|\nga u|<t\}} |D \nga u|^p \phi ^p dx \leq c \int_{B_{3/4}\cap \{|\nga u|<t\}} |\nga u|^p |D\phi|^p dx+ct\int_{B_{3/4}\cap \{|\nga u|\geq t\}} |D \nga u|^{p-1} |D\phi|\phi^{p-1} dx
\]
for all $t>0$.

Next let $\gamma \in (0,1)$.
By multiplying the above inequality by $(1+t)^{-1-\gamma}$, integrating over $(0,\vc)$ with respect to $t$ and then invoking Fubini's theorem we arrive at
\begin{eqnarray*}
\frac{1}{\gamma}\int_{B_{3/4}} \frac{|D \nga u|^p\phi^p}{(1+|\nga u|)^\gamma}dx
& \leq& \frac{c}{\gamma} \int_{B_{3/4}} (1+|\nga u|)^{p-\gamma}|D \phi|^pdx
\\
&& {} +c\int_0^\vc \frac{1}{(1+t)^\gamma} \int_{B_{3/4}\cap \{|\nga u|\geq t\}} |D \nga u|^{p-1}|D \phi|\phi^{p-1} \, dx \, dt.
\end{eqnarray*}
To handle the second integral on the right-hand side of this inequality, an application of Fubini's theorem and Young's inequality gives
\begin{eqnarray*}
c\int_0^\vc \frac{1}{(1+t)^\gamma} \int_{B_{3/4}\cap \{|\nga u|\geq t\}} |D \nga u|^{p-1}|D \phi|\phi^{p-1}dxdt 
&\leq& \frac{c}{1-\gamma} \int_{B_{3/4}} |D \nga u|^{p-1}(1+|\nga u|)^{1-\gamma}|D \phi|\phi^{p-1}dx\\
&\leq& \frac{1}{2\gamma} \int_{B_{3/4}} \frac{|D \nga u|^p \phi^p}{(1+|\nga u|)^\gamma}dx
\\
&& {} +\frac{c\gamma^{p-1}}{(1-\gamma)^{p}} \int_{B_{3/4}}(1+|\nga u|)^{p-\gamma}|D \phi|^p dx.
\end{eqnarray*}	
Hence
\begin{equation}\label{4.54}
\int_{B_{3/4}} \frac{|D \nga u|^p \phi^p}{(1+|\nga u|)^\gamma}dx \leq \frac{c}{(1-\gamma)^{p}} \int_{B_{3/4}}(1+|\nga u|)^{p-\gamma}|D \phi|^p dx.
\end{equation}

From this there are two possibilities.
If $n<p^2$ then $p < q_0$, from which it follows that $u \in L^p(B_{3/4})$.
So taking $\gamma \to 0$ in \eqref{4.54} yields $Du \in L^p(B_{1/2})$.
It remains to consider $p^2\leq n$.
In this case choose $\gamma \geq \frac{n-p^2}{n-p}$.
Using the fact that $\tilde{u} \in W^{1,q}(B_{3/4})$ for all $q \in (1,q_0)$ we deduce that right-hand side in \eqref{4.54} is finite.

Since
\[
\Big|D\big((1+|\nga u|)^{\frac{p-\gamma}{p}}\big)\Big|^p\leq \Big(1-\frac{\gamma}{p}\Big)^p |D \nga u|^p \big(1+|\nga u|\big)^{-\gamma},
\]
\eqref{4.54} implies that
\begin{equation}\label{4.55}
\int_{B_{3/4}}\left|D\Big((1+|\nga u|)^{\frac{p-\gamma}{p}}\phi\Big)\right|^p dx \leq \frac{c}{(1-\gamma)^p} \int_{B_{3/4}}  \big(1+|\nga u|\big)^{p-\gamma}|D \phi|^p \, dx.
\end{equation}
Set $\theta=\frac{n}{n-p}=\frac{p^*}{p}$, where $p^*$ denotes the Sobolev's exponent.
Using Sobolev's inequality and \eqref{4.55}, we obtain
\begin{eqnarray}
\Bigg(\int_{B_{3/4}}\Big((1+|\nga u|)^{1-\gamma/p}\phi\Big)^{\theta p} dx \Bigg)^{1/\theta}
&\le& \frac{c}{(1-\gamma)^p} \int_{B_{3/4}}  \big(1+|\nga u|\big)^{p-\gamma}\phi |D \phi|^p \, dx.
\label{4.56}
\end{eqnarray}

Next we use an iterating argument in the spirit of (finite) Moser's interation to derive the claim.
Define
\[
q_j = \theta^j(p-\gamma), 
\quad \gamma_j = p-q_j,
\quad
B_j = B_{5/8 + 1/(j+1)}
\]
and correspondingly choose $\{\phi_j\}_{j \in \N} \subset \Cc(B^j)$ such that 
\[
0 \leq \phi_j \leq 1,
\quad
\phi_{j+1} \leq  \phi_j
\quad \mbox{and} \quad
\phi_j|_{B_{j+1}}=1
\]
for all $j \in \N$. 
Note that $\{\gamma_j\}_{j \in \N}$ is decreasing. 

Now \eqref{4.56} reads
\[
\Bigg(\int_{B_{3/4}}\big(1+|\nga u|\big)^{\theta(p-\gamma_j)}\phi_j^{\theta p}\Bigg)^{1/\theta} \leq c \int_{B_{3/4}} \big(1+|\nga u|\big)^{p-\gamma_j}|D \phi_j|^pdx
\]
for all $j \in \N$,	provided that $\gamma_j > 0.$ 
In other words $u \in L^{q_j}(B^j)$ implies $u \in L^{q_{j+1}}(B^{j+1})$ for all $j \in \N$ such that $\gamma_j>0$. 

Let $j_0 \in \N$ be the smallest number such that $\gamma_{j_0+1}\leq 0$.
Then $u \in L^{q_{j_0+1}}(B^{j_0+1})$. 
This in particular yields $\nga u \in L^p(B_{5/8})$. 
Combining this with \eqref{4.54} and then taking the limit when $\gamma \to 0$ give $D \nga u \in L^p(B_{1/2})$. 

The claim now follows by reversing the scaling process at the beginning of the proof. 
\end{proof}

The following lemmas are direct consequences of Proposition \ref{prop4.1}.

\begin{lemma}\label{lemma5.2}
Let $\beta \in (1,\infty)$ be such that 
\[
\frac{np}{n-p}<\beta'<\frac{n}{n(2-p)-1}.
\]
Let $B = B_r(x_0)$ be a ball and $f \in L^{\beta}(B).$
Let $u\in W^{1,p}(B)$ be a weak solution to \eqref{5hh070120148} in $B$.
Let $\varepsilon\in (0,1)$ and $q \in (1,q_0)$, where $q_0$ is defined in \eqref{4.2}.
Then there exist $\delta=\delta\left(n,N,p,q,\varepsilon\right)\in (0,1)$ and a $p$-harmonic map $v$ in $\frac{1}{2}B$ such that
\begin{equation}\label{5.6}
\Big(\tb_{\frac{1}{2}B}|Du-Dv|^qdx\Big)^{1/q}\leq \frac{\varepsilon}{r}\tb_{B}|u-(u)_{B_r}|dx+\frac{\varepsilon}{\delta^{1/(p-1)}} \left[r\left(\tb_{B} |f|^{\beta} dx\right)^{1/{\beta}}\right]^{1/(p-1)}.
\end{equation}
\end{lemma}

\begin{proof}
We use a scaling argument with 
\begin{equation}\label{lemma5.3}
\ngang u:=\frac{u-(u)_{B}}{\lambda} 
\quad \mbox{and} \quad 
\ngang f:=\frac{f}{\lambda^{p-1}},
\end{equation}
where
\[
\lambda:=\frac{1}{r}\tb_{B}\left|u-(u)_{B}\right|dx+\left[\frac{r}{\delta}\left(\tb_{B} |f|^{\beta} dx\right)^{1/{\beta}}\right]^{1/(p-1)}
\]
and $\delta=\delta(n,N,p,q,\varepsilon)$ is given in Proposition \ref{prop4.1} with $M=1$.

It follows that
\[
\tb_{B}|\un|dx\leq r
\quad \mbox{and} \quad
-\Delta_p\ngang u=\ngang f \quad \text{in}\,\, B.
\]

If $\lambda = 0$ then $u$ is constant and so we can choose $v= u$.

Next assume that $\lambda>0$.
We have
\[
\Bigg|\tb_{B}|D \un|^{p-2}D\un:D\va dx\Bigg|\leq \frac{1}{\lambda^{p-1}}\left(\tb_{B} |\varphi|^{\beta'} dx\right)^{1/\beta'}\left(\tb_{B} |f|^{\beta} dx\right)^{1/{\beta}}\leq \frac{\delta}{r}\left(\tb_{B} |\varphi|^{\beta'} dx\right)^{1/\beta'},
\]
for all $\varphi\in W^{1,p}_0(B)\cap L^{\beta'}(B)$. Therefore by Proposition \ref{prop4.1} there exists a $p$-harmonic map $\ngang v$ in $\frac{1}{2}B$ such that
\[
\Big(\tb_{\frac{1}{2}B}|D\un-D\ngang v|^qdx\Big)^{1/q}\leq \varepsilon.
\]
Scaling back to $u$ with $v=\lambda \ngang v$ we obtain \eqref{5.6}.
To finish note that $v$ is $p$-harmonic.
\end{proof}

\begin{proposition} \label{inter} 
Adopt the assumptions and notation in Lemma \ref{lemma5.2}.
Then there exist constants $$\delta=\delta\left(n,N,p,q,\varepsilon\right)\in (0,1), C = C(n,p,q)>0$$ and a $p$-harmonic map $v\in W^{1,p}(\frac{1}{2}B)$ such that
\begin{align*}
\left(\tb_{\frac{1}{2}B}|D u-D v|^qdx\right)^{\frac{1}{q}}
\leq \frac{\varepsilon}{\delta^{1/(p-1)}} \left[r\left(\tb_{B} |f|^{\beta} dx\right)^{1/{\beta}}\right]^{1/(p-1)}+  \varepsilon\left(\tb_{B}|D u|^qdx\right)^{1/q}
\end{align*}
and
\begin{align*}
\|D v\|_{L^\infty(\frac{1}{4}B)}\leq \frac{C\varepsilon}{\delta^{1/(p-1)}}\left[r\left(\tb_{B} |f|^{\beta} dx\right)^{1/{\beta}}\right]^{1/(p-1)}
+ C(1+\vae) \left(\tb_{B}|D u|^{q}dx\right)^{1/q}.
\end{align*}
\end{proposition}

\begin{proof} 
Using Lemma \ref{lemma5.2}, Poincare's and Holder's inequalities, there exists a $p$-harmonic map $v\in W^{1,p}(\frac{1}{2}B)$ such that
\[
\Big(\tb_{\frac{1}{2}B}|Du-Dv|^qdx\Big)^{1/q}\leq \vae \left(\tb_{B}|D u|^qdx\right)^{1/q}+\frac{\varepsilon}{\delta^{1/(p-1)}}\left[r\left(\tb_{B} |f|^{\beta} dx\right)^{1/{\beta}}\right]^{1/(p-1)}.
\]

Next it follows from \cite[(3.6)]{KM} that
\[
\|D v\|_{L^\infty(\frac{1}{4}B)}\leq C \tb_{\frac{1}{2}B}|D v|dx \leq C \left(\tb_{\frac{1}{2}B}|D v|^qdx\right)^{1/q}
\]
for a constant $C=C(n,p,q).$

The claim now follows by combining these two estimates together.  
\end{proof}

\section{Good-$\lambda$ type bounds}\label{sec-3}

In this section we present a good-$\lambda$-type estimate - Proposition \ref{5hh23101312}.
In order to do this, we need two auxiliary results.

The first one can be viewed as a (weighted) substitution for the Calderon-Zygmund-Krylov-Safonov decomposition (cf.\ \cite{MP}).

\begin{lemma}\label{phu} Let $\q$ be an $\mathbf{A}_\infty$-weight and $B$ be a ball of radius $R$ in $\R^n$.  Let $E\subset  F \subset B$ be measurable and $\vae \in (0,1)$ satisfy the following property:  
\begin{tabeleq}
\item $\q(E)<\varepsilon \, \q\big(B\big).$
\item $\q(E\cap B_\rho(x))\geq \varepsilon \, \q(B_\rho(x))$ implies $B_\rho(x) \cap B  \subset F$ for all $x\in B$ and $\rho\in (0,R]$.
\end{tabeleq}
Then there exists a $C = C(n,[\q]_{\mathbf{A}_\infty})$ such that $\q(E)\leq C \varepsilon \, \q(F)$.         
\end{lemma}

The next result is a variation of Lemma \ref{phu}.

\begin{lemma}\label{5hhvitali2} 
Let $\q$ be an $\mathbf{A}_\infty$-weight.
Let $E\subset F$ be measurable and $\vae \in (0,1)$ satisfy the following property: For all $x\in \R^n$ and $R \in (0,\vc)$, one has 
\begin{equation}\label{thn3}
\q(E\cap B_R(x))\geq \varepsilon \, \q(B_R(x)) 
\quad \mbox{implies} \quad
B_R(x) \subset F.
\end{equation}
Then there exists a $C = C(n,[\q]_{\mathbf{A}_\infty})$ such that $\q(E)\leq C \varepsilon \, \q(F)$.    
\end{lemma} 

\begin{proof}
Without loss of generality, we may assume that $\q(E) \vee \q(F)<\vc.$ 
Let $x_0 \in \R^n$ and $R$ be sufficiently large such that $\q(E)<\vae \, \q(B_R(x_0))$. Set $S=E \cap B_R(x_0)$ and $T=F \cap B_R(x_0).$ 
The claim follows directly from Lemma \ref{phu} with $S$, $T$, $B_R(x_0)$ and $\vae.$  

Indeed, we have $\q(S)\leq \q(E)<\vae \, \q(B_R(x_0)).$ Assume that $x\in B_R(x_0)$ and $\rho\in (0,R]$ satisfy  
$$\q(S\cap B_\rho(x))\geq \varepsilon \, \q(B_\rho(x)).$$
Obviously we also have
$$\q(E \cap B_\rho(x))\geq \varepsilon \, \q(B_\rho(x)).$$
Therefore \eqref{thn3} implies $B_\rho(x)\subset F,$ from which it follows that $B_\rho(x) \cap B_R(x_0) \subset F \cap B_R(x_0)=T.$

Next Lemma \ref{phu} asserts that there exists a $C = C(n,[\q]_{\mathbf{A}_\infty})$ such that $\q(E \cap B_R(x_0))\leq C \varepsilon \, \q(F \cap B_R(x_0))$.  
Now we let $R$ tend to infinity to complete the proof.
\end{proof}

Recall the maximal function defined by
\[
{\bf M}_\beta(f)(x)
=\sup_{\rho>0} \rho^\beta \, \tb_{B_\rho(x)}|f(y)|dy
\]
for all $x\in\mathbb{R}^{n}$, $f \in L^1_{\mathrm{loc}}(\R^n)$ and $\beta \in [0,n]$.
The case $\beta = 0$ corresponds to the usual Hardy-Littlewood maximal function ${\bf M} = {\bf M}_0$.

We now turn to the aforementioned good-$\lambda$-type estimate.

\begin{proposition}\label{5hh23101312}   
Let $\q\in {\bf A}_\infty$, $\epsilon > 0$ and $q \in (1,q_0)$.
Let $\beta \in (1,\infty)$ be such that $\frac{np}{n-p}<\beta'<\frac{n}{n(2-p)-1}$ and $f\in L^{\beta}(\R^n)$.   
Then there exist constants 
$$
C = C(n,[\q]_{{\bf A}_\infty}),
\quad
\Lambda_0=\Lambda_0(n,p,q)>3^{n/q}
\quad \mbox{and} \quad 
\delta=\delta(n,p,q,\varepsilon,[\q]_{{\bf A}_\infty})\in (0,1),
$$
such that 
\begin{eqnarray*}
&& \q\Bigg[\left\{x\in \R^n: \Big( {\bf M}\big(|D u|^q\big)(x)\Big)^{1/q}>\Lambda_0\lambda,
\,\, \Big(\mathbf{M}_\beta \big(|f|^{\beta}\big)(x)\Big)^{\frac{1}{(p-1)\beta}} \leq \delta^{1/(p-1)} \lambda \right\}\Bigg] 
\\
&\leq& C\varepsilon \, \q \Big(\big\{x\in \R^n: ({\bf M}(|D u|^q)(x))^{1/q}> \lambda\big\}\Big) 
\end{eqnarray*}
for all $\lambda>0$.
\end{proposition} 		

\begin{proof} 
Set 
$$E_{\lambda,\delta}=\left\{y\in \R^n: \Big( {\bf M}\big(|D u|^q\big)(y)\Big)^{1/q}>\Lambda_0\lambda, \,\,
\Big(\mathbf{M}_\beta \big(|f|^{\beta}\big)(y)\Big)^{\frac{1}{(p-1)\beta}} \leq \delta^{1/(p-1)} \lambda \right\}$$ 
and 
$$F_\lambda=\big\{y\in \R^n: \Big({\bf M}(|D u|^q)(y)\Big)^{1/q}> \lambda\big\}$$ for each $\delta\in (0,1)$ and $\lambda>0$. 
Here $\Lambda_0 = \Lambda_0(n,p,q)$ is to be chosen later.  

We will use Lemma \ref{5hhvitali2} for $E_{\lambda,\delta}$ and $F_\lambda$. 
That is, we will verify that 
\[
\q(E_{\lambda,\delta} \cap B_r(x))\geq \varepsilon \q(B_r(x)) \quad \Longrightarrow \quad B_r(x) \subset F_\lambda
\]
for all $x\in \R^n$, $r\in (0,\vc)$ and $\lambda>0$, provided that $\delta$ is sufficiently small.

Indeed, let $x\in \R^n$, $r\in (0,\vc)$ and $\lambda>0$.
To avoid triviality, we consider $E_{\lambda,\delta}\cap B_r(x)\not = \emptyset.$ 
By contraposition, assume that $B_r(x) \cap F^c_\lambda\not= \emptyset.$ Then there exist $x_1,x_2\in B_r(x)$ such that 
\begin{equation} \label{x1x2}
\Big({\bf M}(|D u|^q)(x_1)\Big)^{1/q}\leq \lambda 
\quad \mbox{and} \quad 
\Big(\mathbf{M}_\beta \big(|f|^{\beta}\big)(x_2)\Big) ^{\frac{1}{(p-1)\beta}}\le  \delta^{1/(p-1)} \lambda.
\end{equation}
We aim to show that
\[
\q\big(E_{\lambda,\delta} \cap B_r(x)\big)< \varepsilon \q(B_r(x)). 
\]

First note that 
\begin{equation} \label{maximal truncate}
\Big({\bf M}(|D u|^q)(y)\Big)^{1/q} \leq \max\left\{\Big({\bf M}\left(\1_{B_{2r}(x)}|D u|^{q}\right)(y)\Big)^{\frac{1}{q}},3^{n/q}\lambda\right\}
\end{equation}
for all $y\in B_r(x)$.
Indeed, if $\rho\leq r$ then
\[
\tb_{B_\rho(y)}|Du|^qdx=\tb_{B_\rho(y)} \1_{B_{2r}(x)} |Du|^qdx\leq {\bf M}\left(\1_{B_{2r}(x)}|D u|^{q}\right)(y).
\]
Otherwise $B_\rho(y)\subset B_{2r+\rho}(x_1)$ and we have
\[
\tb_{B_\rho(y)}|Du|^qdx\leq \frac{1}{|B_\rho(y)|}\tp_{B_{3\rho}(x_1)}|Du|^qdx=3^n\tb_{B_{3\rho}(x_1)}{\bf M}(|D u|^q)(x_1) \leq 3^n \lambda^q.
\]

It follows from \eqref{maximal truncate} that 
\[
E_{\lambda,\delta} \cap B_r(x)=\left\{y \in \R^n: \Big({\bf M}\left(\1_{B_{2r}(x)}|D u|^{q}\right)(y)\Big)^{\frac{1}{q}}>\Lambda_0\lambda, \Big(\mathbf{M}_\beta\big(|f|^{\beta}\big)(y)\Big)^{\frac{1}{(p-1)\beta}}\leq \delta^{1/(p-1)} \lambda  \right\} \cap B_r(x)
\]
for all $\lambda>0$ and $\Lambda_0\geq 3^{n/q}$.

Applying  Proposition \ref{inter} to  $u \in W_{0}^{1,p}(\R^n), f, B=B_{8r}(x)$ and $\eta \in (0,1)$, there exist constants $\delta = \delta(n,p,q,\varepsilon,[\q]_{\mathbf{A}_\infty})\in (0,1)$, $C_0 = C_0(n,p,q) > 0$ and a $p$-harmonic map $v \in W^{1,p}(B_{4r}(x))$ such that
\begin{align*}
\|D v\|_{L^\infty(B_{2r}(x))}
\leq \frac{C_0\eta}{\delta^{1/(p-1)}} \left[r\left(\tb_{B_{8r}(x)} |f|^{\beta} dy\right)^{1/{\beta}}\right]^{1/(p-1)}+C_0(1+\eta) \left(\tb_{B_{8r}(x)}|D u|^{q}dy\right)^{1/q}
\end{align*}
and
\begin{align*}
\left(\tb_{B_{4r}(x)}|D u-D v|^qdx\right)^{\frac{1}{q}}
\leq \frac{\eta}{\delta^{1/(p-1)}} \left[r\left(\tb_{B_{8r}(x)} |f|^{\beta} dx\right)^{1/{\beta}}\right]^{1/(p-1)}+  \eta \left(\tb_{B_{8r}(x)}|D u|^qdx\right)^{1/q}.
\end{align*}

Using \eqref{x1x2} we deduce that
\begin{eqnarray}
\|D v\|_{L^\infty(B_{2r}(x))} 
&\leq&   \frac{C_0\eta}{\delta^{1/(p-1)}}\Big(\mathbf{M}_{\beta}\big(|f|^{\beta}\big)(x_2)\Big)^{\frac{1}{(p-1)\beta}} +C_0(1+\eta)\left[{\bf M}(|D u|^q)(x_1)\right]^{1/q}  
\nonumber
\\
&\leq& C_0(1+\eta)\lambda \leq 2C_0\lambda
\label{DvLinf}
\end{eqnarray}
and 
\begin{eqnarray}
\left(\tb_{B_{4r}(x)}|D u-D v|^qdx\right)^{\frac{1}{q}}
&\leq& \frac{\eta}{\delta^{1/(p-1)}}\left[R\left(\tb_{B_{8r}(x)} |f|^{\beta} dx\right)^{1/{\beta}}\right]^{1/(p-1)}+  \eta \left(\tb_{B_{8r}(x)}|D u|^qdx\right)^{1/q}
\nonumber
\\
&\leq& \frac{\eta}{\delta^{1/(p-1)}} \Big(\mathbf{M}_{\beta}\big(|f|^{\beta}\big)(x_2)\Big)^{\frac{1}{(p-1)\beta}} +\eta \left[{\bf M}(|D u|^q)(x_1)\right]^{1/q}          
\nonumber
\\
&\leq& \eta \lambda.
\label{Dvdiff}
\end{eqnarray}                                                            

Clearly
\begin{align*}
\left[	{\bf M}\Big(\Big|\sum_{j=1}^{3}f_j\Big|^q\Big)\right]^{1/q}\leq 3\sum_{j=1}^{3}	\left[	{\bf M}\Big(\big|f_j\big|^q\Big)\right]^{1/q}.
\end{align*}       
Hence
\begin{eqnarray}                  
|E\cap B_r(x)|
&\leq&   \big|\{y\in \R^n: {\bf M}\left(\1_{B_{2r}(x)}|D (u-v)|^q(y)\right)^{\frac{1}{q}}>\Lambda_0\lambda/9\}\cap B_r(x) \big|
\nonumber
\\
&& {} + \big| \{y\in \R^n: {\bf M}\left(\1_{B_{2r}(x)}|Dv|^q(y)\right)^{\frac{1}{q}}>\Lambda_0\lambda/9\}\cap B_r(x) \big|.   
\label{es18}      
\end{eqnarray}

In view of \eqref{DvLinf} there holds 
\begin{align*}
\big|y \in \R^n: \{{\bf M}\left(\1_{B_{2r}(x)}|Dv|^q(y)\right)^{\frac{1}{q}}>\Lambda_0\lambda/9\}\cap B_r(x)\big|=0,
\end{align*}
provided that
$\Lambda_0\geq \max\{3^{n/q},30C_0\}$.

Combining \eqref{Dvdiff} and \eqref{es18} yields
\begin{align*}
|E\cap B_r(x)|
&\leq   \big|\Big\{y\in \R^n:\left[{\bf M}\left(\1_{B_{2r}(x)}|D (u-v)|^q(y)\right)\right]^{\frac{1}{q}}>\Lambda_0\lambda/9\Big\}\cap B_r(x)\big|\\
& \quad \leq \frac{C}{\lambda^q}\int_{B_{2r}(x)}  |D (u-v)|^q dx
\leq  C \eta^q r^n,
\end{align*} 
where we used the fact that ${\bf M}$ is of weak type $(1,1)$ in the second step.

Thus 
\begin{align*}
\q(E\cap B_r(x))
\leq c\left(\frac{|E\cap B_r(x) |}{|B_r(x)|}\right)^\nu \q(B_r(x))
\leq  c(C\eta^q)^{\nu} \q(B_r(x))
< \varepsilon \q(B_r(x)),
\end{align*} 
where we chose $\eta$ small enough such that $c(C\eta^q)^{\nu}<\epsilon.$     

This completes our proof.                 
\end{proof}

\section{Global weighted gradient estimates} \label{main proof}

With the knowledge from the previous sections, we are now ready to tackle the main theorem.

\begin{proof}[{\bf Proof of Theorem \ref{101120143-p}}] 
By Theorem \ref{5hh23101312}, for all $\varepsilon>0$ and $q \in (1,q_0)$, where $q_0$ is defined in \eqref{4.2} there exist constants $C = C(n,[\q]_{{\bf A}_\infty})$, $\delta=\delta(n,p,q,\varepsilon,[\q]_{{\bf A}_\infty})\in (0,1)$ and 
$\Lambda_0=\Lambda_0(n,p,q)>3^{n/q}$ such that 
\begin{align*}
& \q\left(\Big\{x\in\R^n:({\bf M}(|D u|^q)(x))^{1/q}>\Lambda_0\lambda, \Big(\mathbf{M}_{\beta}\big(|f|^{\beta}\big)(x)\Big)^{\frac{1}{(p-1)\beta}}\le \delta^{1/(p-1)} \lambda \Big\}\right)\\
& \qquad \qquad \leq C\varepsilon \q\left(\left\{x\in\R^n: \Big({\bf M}(|D u|^{q})(x)\Big)^{1/q}> \lambda \right\}\right)
\end{align*}
for all $\lambda>0$.

By hypothesis $\Phi$ is invertible and $\Phi^{-1}:[0,\infty)\rightarrow [0,\infty)$.
Therefore
\begin{eqnarray*}
\q\left(\left\{x\in\R^n:\left[{\bf M}(|D u|^q)(x)\right]^{1/q}>\Phi^{-1}(t)\right\}\right)
&\le& \q\left(\left\{x\in\R^n: \Big(\mathbf{M}_{\beta}\big(|f|^{\beta}\big)(x)\Big)^{\frac{1}{(p-1)\beta}}> \frac{\delta^{1/(p-1)}}{\Lambda_0} \Phi^{-1}(t) \right\}\right)
\\
&& {} + C\varepsilon \q\left(\left\{x\in \R^n: \Big({\bf M}(|D u|^q)(x)\Big)^{1/q}> \frac{\Phi^{-1}(t)}{\Lambda_0}\right\}\right)
\end{eqnarray*}
for all $t>0$. 
This in turn implies
\begin{eqnarray*}
&&\int_{0}^{T}\q \left(\left\{x\in\R^n:	\Phi \left[\Big({\bf M}(|D u|^q)(x)\Big)^{\frac{1}{q}}\right]>t\right\}\right) dt
\\
&\leq& C\varepsilon \int_{0}^{T}\q \left(\left\{x\in\R^n:	\Phi\left[\Lambda_0\Big({\bf M}(|D u|^{q})(x)\Big)^{\frac{1}{q}}\right]>t\right\}\right) dt \\
&& {} +   \int_{0}^{T} \q\left(\left\{x\in\R^n: \Phi\left[\frac{\Lambda_0}{\delta^{1/(p-1)}}\Big(\mathbf{M}_{\beta}\big(|f|^{\beta}\big)(x)\Big)^{\frac{1}{(p-1)\beta}}\right]> t  \right\}\right)dt
\\
&\leq& C\varepsilon \int_{0}^{T}\q\left(\left\{x\in\R^n:	 H_1\Phi\left[\Big({\bf M}(|D u|^{q})(x)\Big)^{\frac{1}{q}}\right]>t\right\}\right) dt \\
&& {} + \int_{0}^{T} \q\left(\left\{x\in\R^n: H_2\Phi\left[\Big(\mathbf{M}_{\beta}\big(|f|^{\beta}\big)(x)\Big)^{\frac{1}{(p-1)\beta}}\right]> t \right \}\right)dt,
\end{eqnarray*}
where we used the fact that $\Phi(2t)\leq c\, \Phi(t)$ and $\Phi$ is increasing in the second step.
Here $T > 0$, $H_1=c^{\lceil\log_{2}(\Lambda_0)\rceil}$ and  $H_2=c^{\left\lceil\log_{2}\left(\frac{\Lambda_0}{\delta^{1/(p-1)}}\right)\right\rceil}$, in which $\lceil \cdot \rceil$ denotes the ceiling function.  

Using a change of variables we arrive at 
\begin{eqnarray*}
&&\int_{0}^{T}\q\left(\left\{x\in\R^n:	\Phi\left[\Big({\bf M}(|D u|^{q})(x)\Big)^{\frac{1}{q}}\right]>t\right\}\right) dt
\\
&\leq& H_1 C\varepsilon \int_{0}^{\frac{T}{H_1}}\q \left(\left\{x\in\R^n:	 \Phi\left[\Big({\bf M}(|D u|^{q})(x)\Big)^{\frac{1}{q}}\right]>t\right\}\right) dt \\
&& {} + H_2 \int_{0}^{\frac{T}{H_2}} \q\left(\left\{x\in\R^n: \Phi\left[\Big(\mathbf{M}_{\beta}\big(|f|^{\beta}\big)(x)\Big)^{\frac{1}{(p-1)\beta}}\right]> t  \right\}\right)dt.
\end{eqnarray*}

Now we choose $\varepsilon=\frac{1}{2H_1 C}$ so that the first integral on the right is absorbed by the left-hand term, which yields
\begin{eqnarray*}
&&\int_{0}^{T}\q \left(\left\{x\in\R^n:	\Phi\left[({\bf M}(|D u|^{q})(x))^{\frac{1}{q}}\right]>t\right\}\right) dt
\\
&\leq& 	2 H_2 \int_{0}^{\frac{T}{H_2}} \q\left(\left\{x\in\R^n: \Phi\left[\Big(\mathbf{M}_{\beta}\big(|f|^{\beta}\big)(x)\Big)^{\frac{1}{(p-1)\beta}}\right]> t  \right\}\right)dt.
\end{eqnarray*}
Recall that
\begin{align*}
\int_{\R^n}\Phi(|f|) \q dx= \int_{0}^{\infty}\q(\{x\in\R^n:\Phi(|f(x)|)>t\}) dt.
\end{align*}	
Thus by letting $T\rightarrow \infty$ in the above inequality we arrive at 
\begin{equation*}
\int_{\R^n}\Phi\left[\Big({\bf M}\left(|D u|^q\right)\Big)^{\frac{1}{q}}\right] \q dx\leq 2H_2 \int_{\R^n}\Phi\left[\Big(\mathbf{M}_{\beta}\big(|f|^{\beta}\big)\Big)^{\frac{1}{(p-1)\beta}}\right] \q \, dx
\end{equation*}
as required.
\end{proof}   


\end{document}